\documentclass[12pt]{amsart}
\usepackage{mathptmx,eucal,amsmath,amscd,amssymb,amsthm,xspace, diagrams}
\usepackage{hyperref,amssymb,color, url,fancyhdr}
\usepackage[cp850]{inputenc}
\usepackage{latexsym}
\usepackage{longtable}
\usepackage{graphics}
\usepackage{array}

\definecolor{red}{rgb}{1,0,0}

\theoremstyle{plain}

\newtheorem{thm}{Theorem}[section]

\newtheorem{cor}[thm]{Corollary}

\newtheorem{conj}{Conjecture}

\newtheorem{quest}[thm]{Question}
\newtheorem{pro}[thm]{Proposition}

\newtheorem{lem}[thm]{Lemma}
\newtheorem{observ}[thm]{Observation}

\newtheorem*{namedtheorem}{\theoremname}
\newcommand{\theoremname}{testing}
\newtheorem{fact}[thm]{Fact}
\newenvironment{named}[1]{\renewcommand{\theoremname}{#1}\begin{namedtheorem}}{\end{namedtheorem}}

\theoremstyle{definition}
\newtheorem*{defn}{Definition}

\newtheorem{exa}[thm]{Example}
\newtheorem{rem}[thm]{Remark}

\newtheorem{claim}[thm]{Claim}

\def\Z{\mathbb{Z}}

\def\F{\mathbb{F}}

\def\hat{\widehat}

\DeclareMathOperator{\GL}{GL}\DeclareMathOperator{\PSL}{PSL}
\DeclareMathOperator{\SL}{SL}\DeclareMathOperator{\Aut}{Aut}
\DeclareMathOperator{\ssl}{\mathfrak{sl}}

\DeclareMathOperator{\lcm}{lcm}

\newcommand{\FF}{\mathcal{F}}

\newcommand{\intfunc}{i}

\newcommand{\asymptotic}{ \dot{\ \sim \ }}%
\newcommand{\asymptoticless}{ \dot{\  \preceq \ }}
\newcommand{\asymptoticge}{\dot{\ \succeq \ }}

\newcommand{\congruent}{{\  \sim \ }}%
\newcommand{\congruentleq}{ {\  \preceq \ }}
\newcommand{\congruentgeq}{{\ \succeq \ }}

\newcommand{\intersection}{\cap}

\newcommand{\BN}{\mathbb N} 
 \newcommand{\BZ}{\mathbb Z}

\newcommand{\CG}{\mathcal G} %\newcommand{\CH}{\mathcal H}

\newcommand{\CO}{\mathcal O}

\DeclareMathOperator{\automorphism}{Aut}
\DeclareMathOperator{\outerautomorphism}{Out}
\DeclareMathOperator{\Out}{Out}

\newcommand{\comment}[1]{}

\fancypagestyle{plain}

\begin{document}
\bibliographystyle{plain}

%-----------------------------------------------------------
%-----------------------------------------------------------

\title{\textbf{Intersection growth in groups}}
\author{Ian Biringer, Khalid Bou-Rabee, Martin Kassabov, Francesco Matucci}
\maketitle

%-----------------------------------------------------------
%-----------------------------------------------------------
%---------------------Abstract---------------------------

\begin{abstract}
The intersection growth of a group $G$ is the asymptotic behavior of the index of the intersection of all subgroups of $G $ with index at most $n$, and measures the Hausdorff dimension of $G$ in profinite metrics.
We study intersection growth in free groups
and special linear groups and relate intersection growth to quantifying residual finiteness.
\end{abstract}

\smallskip\smallskip

%\nid keywords: \emph{The people's keywords.}

%-----------------------------------------------------------
%-----------------------------------------------------------
\section{Introduction}
A group $G $ is called \emph {residually finite} if for every nontrivial element $g \in G $ there is a homomorphism $\phi : G \to F$ onto a finite group with $\phi(g) \neq 1 $.  A subtle related problem is to determine how many elements of $G $ can be detected as nontrivial in \emph{small} finite quotients $F$, i.e.\ those with cardinality at most some $n $.  This problem is known as \emph{quantifying residual finiteness}, and has been studied in ~\cite{Bou}, ~\cite{Buskin}, ~\cite{BM1}, ~\cite{BM2}, ~\cite{KM11}.   In these papers, the idea is to fix a generating set $S$ for $G $ and to determine the size $F_G^S(r)$ of the largest finite quotient needed to detect as nontrivial an element of $G$ that can be written as an $S$-word with length at most $r $.  Fine asymptotic bounds for this \emph{residual finiteness growth function} $F_G^S(r)$ are given for a number of groups, in particular free groups, and a closely related function is shown to characterize virtual nilpotence in ~\cite{BM1}.

In this article, we study instead the \emph {percentage} of elements of $G $ that can be detected as nontrivial in a quotient of size $n $.  Specifically, the (normal) \emph {intersection growth function} $i^\lhd_G(n)$ of $G$ is the index of the intersection of all normal subgroups of $G $ with index at most $n $.  In addition to its relation to the program above, this function has geometric motivation: we show in Section \ref{sec:profinite} that intersection growth is a profinite invariant and that its asymptotics control the Hausdorff dimension of the profinite completion of $G$.  

The majority of this paper concerns bounds for variants of $i^\lhd_G(n)$ in special linear groups and free groups, which we will state precisely in the next section.  However, in Section \ref{sec:applications}, we also explain how intersection growths can be used to extract information about the residual finiteness growth
function and identities in groups. Moreover, in an upcoming
work by the authors a fine analysis of $i^\lhd_G(n)$ will be given
for nilpotent groups, mirroring the work of Grunewald, Lubotzky, Segal,
and Smith  ~\cite{LS03} on \emph{subgroup growth}, which counts the number
of subgroups of index at most $n$ in a group.

\section{Definitions and statements of main results \label{sec:statements}}

Let $\CG$ be a class of subgroups of a group $\Gamma$.
We define the $\CG$-\emph{intersection growth function} of $\Gamma $ by letting $\intfunc^\CG_\Gamma(n)$ be the index of the intersection of all $\CG$-subgroups of $\Gamma $ with index at most $n $.   In symbols,
$$
\intfunc^\CG_\Gamma(n) := [\Gamma : \Lambda^\CG_\Gamma(n)]
,\ \ \  \text { where } \ \Lambda^\CG_\Gamma(n) := \bigcap_{[\Gamma: \Delta] \leq n, \Delta \in \CG} \Delta.
$$
Here, $\CG $ will always be either the class of all subgroups, the class $\lhd$ of normal subgroups, the class $\max$ of maximal subgroups or the class $\max \lhd $ of maximal normal subgroups of $\Gamma $, i.e.\ those subgroups that are maximal among normal subgroups.  The corresponding intersection growth functions will then be written $i_{\Gamma}  (n)$, $i_{\Gamma} ^\lhd (n) $, $i_{\Gamma} ^ {\max} (n) $
and $i_{\Gamma} ^ {\max\lhd} (n) $.
%The former clearly grows at least as fast as the latter, but in general their asymptotics can be very different.

Our main theorem is a precise asymptotic calculation of the maximal normal intersection growth and the maximal intersection growth of free groups.

\begin {named}{Theorem~\ref{maximal}}
Let $\FF ^ k $ be the rank $k $ free group. Then we have $$ i_{\FF ^ k}^ {\max \lhd} (n) \asymptotic  e^{{ n ^ { k -\frac{2}{3}} } }\  \text{ and } \  i_{\FF ^ k}^ {\max} (n) \asymptotic
i_{\FF ^ k} (n) \asymptotic e^{{n^n}} .$$
\end {named}

%\marginpar{it will be nice if we can get a result about $\max$ intersection growth -- this should boil down to some estimates about permutation groups.}

\medskip
Here we write $f(n) \asymptotic g(n)$ if there exist suitable constants $A,B,C,D>0$ such that
$f(n) \le Ag(Bn)$ and $g(n) \le Cf(Dn)$ for all positive integers $n$.  In the proof, we use the classification theorem for finite simple groups to show that the maximal normal intersection growth of $\FF^k$ is controlled by subgroups with quotient isomorphic to
$\PSL_2 (p)$, whereas the maximal intersection growth comes from alternating groups. Note that Theorem~\ref {maximal} clearly gives a lower bound for the normal intersection growth of $\FF ^ k $. 

We can also calculate the intersection growth of special linear groups.

\begin{named}{Theorem \ref{SL_3}}
For the special linear groups $\SL_k(\Z)$, where $k\geq 3$, we have 
$$  i_{\SL_k(\Z)}^\lhd(n) \asymptotic  i_{\SL_k(\Z)}^{\max \lhd}(n) \asymptotic
 e^{n^{1/(k^2-1)}}, \ \text{ but } \ 
 i_{\SL_k(\Z)}^{\max}(n) \asymptotic  e^{n^{1/(k-1)}}.$$
\end{named}

The intersection growth of a group is comparable to that of its finite index subgroups (see Lemma \ref{finiteindex}).  So as $\SL_2(\Z)$ is virtually free, the asymptotics in the $k=2$ case are wildly different from those in Theorem \ref{SL_3}.

\subsection{Acknowledgements}
The first author was partially supported by NSF Postdoctoral Fellowship DMS-0902991.
The second author was partially supported by NSF RTG grant DMS-0602191 and the Ventotene 2013 conference.
The third author was partially supported by NSF grants DMS-0900932 and DMS-1303117.
The fourth author gratefully acknowledges the Fondation Math\'ematique
Jacques Hadamard
(FMJH - ANR - Investissement d'Avenir) for the support received during the development of this work.
Finally, we are grateful to Benson Farb, Alex Lubotzky, Ben McReynolds, Peter Neumann
and Christopher Voll for helpful mathematical conversations and references.

\section {Notation and basic properties of intersection growth} \label {basics}

We introduce here some asymptotic notation and study the relationship between the intersection growth of a group and of its subgroups.  We say
$$
a(n) \asymptoticless b(n) \ \ \ \text { if } \ \ \ \exists \,C,D >0 \text { such that } \forall n, \ a(n) \leq C \,b(Dn).
$$
Similarly, $a(n) \asymptotic b (n) $ means that both $a(n) \asymptoticless b (n) $ and $b(n) \asymptoticless a (n) $.  Sometimes we will have sharper asymptotic control, in which case we write
$$
a(n) \congruentleq b(n) \ \ \ \text { if } \ \ \ \limsup_{n\to\infty} \frac {a(n) } {b(n)} \leq 1.
$$
Then, as before, $a(n) \congruent b (n) $ means that both $a(n) \congruentleq b (n) $ and $b(n) \congruentgeq a (n) $.

\begin{lem}\label{finiteindex}
Let k be a natural number and $\Delta$ an index k subgroup of $\Gamma$. Then
\begin {itemize}
\item $i_\Gamma (n) \leq k \cdot \intfunc_\Delta (n) \leq  \intfunc_\Gamma(k n)$, so we have $$\intfunc_\Delta(n) \asymptotic \intfunc_\Gamma(n).$$
\item $i_\Gamma ^\lhd(n) \leq k \cdot \intfunc^{\lhd}_\Delta (n) \leq \intfunc^{\lhd}_\Gamma((kn)^{k}),$ so we have
$$i_\Gamma ^\lhd (n) \asymptoticless \intfunc^{\lhd}_\Delta (n) \asymptoticless \intfunc^{\lhd}_\Gamma(n^{k}).$$
\end {itemize}
\end{lem}
\begin{proof}
For the first part, note that an index $n $ subgroup of $\Delta $ is an index $kn$ subgroup of $\Gamma $. This shows that
$
\Lambda_\Delta(n) \geq \Lambda_\Gamma( kn ).
$
Moreover, if $H \leq \Gamma $ then $[\Gamma: H] \geq [\Delta:\Delta\cap H]$.  From this we obtain
that
$$
\Lambda_\Gamma (n) \ =
 \bigcap_{H\leq \Gamma, \, [\Gamma:H]\leq  n} H \ \ \ge \ \bigcap_{H \leq \Gamma, \,[\Delta:\Delta \cap H]\leq n} \Delta \cap H
\ge \ \Lambda_\Delta(n).
$$
The first item of the lemma then follows since
$$
\Lambda_\Gamma(k n) \le \Lambda_\Delta (n) \le \Lambda_\Gamma (n).
$$
The first inequality of the second item follows exactly as above, since intersecting an index $n $ normal subgroup of $\Gamma $ with $\Delta $ gives an index at most $n $ normal subgroup of $\Gamma $.  The second inequality, however, is different since normal subgroups of $\Delta $ are not necessarily normal in $\Gamma $.

So, suppose that $\Delta $ has coset representatives $g_1, \ldots, g_k$.
For any normal subgroup $N$ of $\Delta$, we have
$$
\bigcap_{i=1}^k g_i N g_i^{-1}
$$
is normal in $\Gamma$ and has index at most $([\Gamma:N])^k$. Hence,
$
\Lambda^{\lhd}_\Delta(n) \geq \Lambda^{\lhd}_\Gamma( (kn)^k ).
$
Further,
$$
[\Gamma : \Lambda^{\lhd}_\Gamma( (kn)^k )] \geq [\Gamma: \Lambda^{\lhd}_\Delta(n)] = k[\Delta: \Lambda^{\lhd}_\Delta(n)].
$$
So $\intfunc_\Gamma((kn)^k) \geq k \intfunc_\Delta(n)$.
\end{proof}

We will soon see in Proposition~\ref{Zcase}
and Theorems~\ref{SL_3} and
~\ref{maximal} that
$$
 i_{\BZ} ^\lhd(n) \congruent e^n, \ \  i_{\FF ^ 2} ^\lhd(n) \asymptoticge  e ^ {n ^ {\frac 43}}, \ \ i_{\SL_3 (\BZ)} ^\lhd(n) \asymptotic e^{n^\frac 18}.
$$
Since there are inclusions $\Z \leq \FF ^ 2 \leq \SL_3(\Z)$,
for infinite index subgroups there is no general relationship between containment and intersection growth.

Here is an example showing the necessity of the power $n^k $ in the second half of Lemma~\ref{finiteindex}.

\begin{exa}\label {qex}
Let $Q < \GL_2(\BZ)$ be the order $8 $ subgroup generated by
$$
\begin {pmatrix} 0 & - 1\\1 & 0 \end {pmatrix} \text { and }\begin {pmatrix} 0 & 1\\1 & 0 \end {pmatrix}.
$$
One can compute that
%(We claim that if )
$G=\BZ ^ 2 \rtimes Q$
has
%then we have
$ i^ {\lhd}_G(n) \asymptotic e^{\sqrt n}$.  Since $ i ^{\lhd}_{\BZ ^2} (n) \asymptotic e ^ n$, this shows that the \emph {normal} intersection growth may indeed increase upon passing to a finite index subgroup, as is allowed by Lemma~\ref{finiteindex}.
The difference comes from the fact that $Q$ acts irreducibly on $(\BZ/p\BZ)^2$ for $p \geq 2$.\footnote{We do not know if the exponent $k$ in part b) of the lemma is the best possible -- this example can be extended to show that exponent must grow with the index.}

In fact, one can also prove that $\log i_\Gamma(n) \congruentleq n$ while $\log i_{\Z^ 2} (n) \congruent 2n$, which shows the necessity of a factor like $k$ in $i_\Gamma(kn)$ in the first part of the lemma.  The point is that the subgroups $i\Z \times \BZ$ and $\BZ \times i\BZ$, $i\leq n$, of $\Z^2$ that one intersects to realize $i_{\BZ^2}(n)$ cannot themselves be realized as intersections $\Delta \cap \BZ ^ 2 $ of subgroups $\Delta <\Gamma $ with $[\Gamma : \Delta]=i$.  This contrasts with the case of the product $\BZ^2 \times Q$, wherein any subgroup $\Delta < \BZ^2$ is the intersection with $\BZ^2 $ of $\Delta \times Q$, a subgroup of the product with the same index as $\Delta$ had in $ \BZ^2$.
\iffalse
To see this, assume that $H \leq \BZ ^ 2 $ is a finite index subgroup that is normal in $G $,  i.e.\ it is invariant under the action of $Q \curvearrowright \BZ^2$. Using the observation that the two dimensional representation of $Q$ is irreducible over $\F_p$ for $p\geq 3$ one see that   $$
H =n \BZ ^ 2 \quad \text{or} \quad \mbox{span}\{(n,n), (2n,0)\} \quad\text { for some } n\in\BZ .
$$

Every index at most $n$ normal subgroup of $G $ intersects $\BZ ^ 2 $ in a $Q $-invariant index at most $n $ subgroup.  This implies that
$$
\Lambda_G^\lhd(n) \ \ \supset \bigcap_{Q\text {-inv} \  H \leq \BZ ^ 2, [\BZ ^ 2: H]\leq n} H \ \supset\  \lcm(1,\ldots, \lfloor \sqrt 2n \rfloor) \BZ ^ 2. $$
Also, every $Q $-invariant index at most $\frac{n}{8} $ subgroup of $\BZ ^ 2 $ is an index at most $n $ normal subgroup of $G $.  Therefore,
$$
\lcm(1,\ldots, \lfloor \sqrt {\frac{n}{8}} \rfloor) \BZ ^ 2 \ \supset \ \bigcap_{Q\text {-inv} \ H \leq \BZ ^ 2, [\BZ ^ 2: H]\leq \frac n8} H \ \supset \ \ \Lambda_G^\lhd(n).
$$
Therefore, $e^{\sqrt n} \asymptotic \lcm(1,\ldots, \lfloor \sqrt {\frac{n}{8}} \rfloor) ^ 2 \leq i_G ^\lhd (n) \leq \lcm(1,\ldots, \lfloor \sqrt {2n} \rfloor) ^ 2 \asymptotic e ^ {\sqrt n}. $
\fi
\end{exa}

Intersection growth behaves well with respect to direct products:

\begin{pro}
\label{pro:products}
\begin{enumerate}
\item if $\Gamma= \Delta_1 \times \Delta_2$ then $i_\Gamma^\bullet(n) = i_{\Delta_1}^\bullet(n) . i_{\Delta_2}^\bullet(n)$ where $\bullet$ is one of
$\max,\lhd,\max\lhd$ or $\bullet$ or no symbol at all.
\item if $\Gamma= \prod_{s=1}^\infty \Delta_s$ then $i_\Gamma^\bullet(n) = \prod_s i_{\Delta_s}^\bullet(n)$ provided that $i_{\Delta_s}^\bullet(n)=1$ for almost all $s$.
\end{enumerate}
\end{pro}
\begin{proof}
Part (2) is an immediate corollary of part (1). We show (1) for intersection growth,
the other cases being similar.
Observe that for any
$H$ has index $\le n$ in $\Gamma_1 \times \Gamma_2$, we have
\[
[\Gamma_i:\Gamma_i \cap H]\le [\Gamma_1 \times \Gamma_2 : H]\le n.
\]
If $K$ has index $\le n$ in $\Gamma_2$, it is clear that
$\Gamma_1K$ has the same index in $\Gamma_1 \times \Gamma_2$.
Keeping this in mind, we compute the following string of containments:
\begin{align*}
\Gamma_1 \Lambda_{\Gamma_2}(n)=\bigcap_{\substack{K \le \Gamma_2 \\ [\Gamma_2:K]\le n}} \Gamma_1 K
\ge \Lambda_{\Gamma_1 \times \Gamma_2}(n)=
\\
\bigcap_{\substack{H \le \Gamma_1 \times \Gamma_2 \\ [\Gamma_1 \times \Gamma_2:H]\le n}} H
\ge
\bigcap_{\substack{H \le \Gamma_1 \times \Gamma_2 \\ [\Gamma_1 \times \Gamma_2:H]\le n}} H\cap \Gamma_2 \ge
\bigcap_{\substack{S \le \Gamma_2 \\ [\Gamma_2:S]\le n}}S=\Lambda_{\Gamma_2}(n).
\end{align*}
Similarly one has that $\Lambda_{\Gamma_1}(n)\Gamma_2 \ge \Lambda_{\Gamma_1 \times \Gamma_2}(n)
\ge \Lambda_{\Gamma_1}(n)$. Thus one has that
\[
\Lambda_{\Gamma_1}(n)\Lambda_{\Gamma_2}(n) = \Lambda_{\Gamma_1}(n)\Gamma_2
\cap \Gamma_1\Lambda_{\Gamma_2}(n) \ge \Lambda_{\Gamma_1 \times \Gamma_2}(n)
\ge \Lambda_{\Gamma_1}(n)\Lambda_{\Gamma_2}(n)
\]
and the result follows.
\end{proof}

For quotients, the correspondence theorem always yields a lower bound.
\begin{observ}
\label{thm:quotients}
Let $N$ be a normal subgroup of $\Gamma$. Then
$$i ^\bullet_{\Gamma/N}(n) \le i^\bullet_{\Gamma}(n).$$
\end{observ}

For extensions, one still has upper bounds.

\begin {pro}
\label{pro:quotients}
Suppose that $1 \longrightarrow N \longrightarrow \Gamma \longrightarrow Q \longrightarrow 1$ is exact.  Then 
$$i^\bullet_{\Gamma}(n) \leq i^\bullet_{N} (n) \cdot i ^\bullet_{Q}(n),$$
where $\bullet $ is $\lhd$ or no symbol at all.  The same bound holds for $\max,\max\lhd$ when the extension is split.
\end {pro}
\begin {proof}
This follows from the fact that if $\Delta \subset \Gamma $ is a subgroup and $\Delta_N, \Delta_Q $ are its intersection with $N $ and projection to $Q$, then we have
$$\max([\Delta:\Delta_N], [Q:\Delta_Q]) \ \leq \ [\Gamma: \Delta] \ \leq \ [\Delta:\Delta_N]\cdot [Q:\Delta_Q].$$
The split assumption is used in the latter two cases to show that if $\Delta$ is maximal or maximal normal in $\Gamma $, then so are $\Delta_N < N$ and $ \Delta_Q <Q.$
\end {proof}

%some comments about profinite completions...

\section {The profinite perspective}
\label {sec:profinite}

If $\Gamma $ is a finitely generated group, its \emph{profinite completion} $\widehat \Gamma$ is the inverse limit of the system of finite quotients of $\Gamma $, taken in the category of topological groups.  In fact, intersection growth is really a profinite invariant, in that it only depend on the profinite completion of the group $\Gamma$.
\begin{lem}
\label{pro-finite}
Let $\Gamma$ be a finitely generated group  and let $\widehat{\Gamma}$ be its profinite completion.
Then
$
i^\bullet_{\Gamma}(n) = i^\bullet_{\widehat{\Gamma}}(n)
$
for every positive integer $n$, where $\bullet$ is one of
$\max,\lhd,\max\lhd$  or no symbol at all.
\end{lem}
\begin{proof}
This follows from the observation that there is a bijection between finite index subgroups of $\Gamma$ and finite index closed subgroup of $\widehat{\Gamma}$ which preserves intersections. In the definition of $ i^\bullet_{\widehat{\Gamma}}$ we can either take all finite index subgroups or all closed finite index subgroups, as by a result of Nikolov and Segal
\cite{nikolov-segal-1}
all finite index subgroups in a topologically finitely generated profinite groups are closed, so there is no difference in our case.
\end{proof}

The profinite completion of $\Gamma$ can also be considered as a metric completion. Namely, a \it profinite metric \rm on $\Gamma $ is defined by fixing a decreasing function $\rho: [0,\infty) \to [0,\infty)$ with $\lim_{x\to\infty}\rho(x)=0$ and letting
$$d_\rho(g,h) := \rho \left(\min \{ [\Gamma : \Delta ] : \Delta \lhd \Gamma, g h^{-1} \notin \Delta \}\right). $$
When $\Gamma $ is residually finite, $d_\rho $ is a metric on $\Gamma$, while in general it defines a metric on the quotient of $\Gamma $ by the intersection of its finite index subgroups.  In any case, the metric completion of this space is homeomorphic to $\hat\Gamma $.

Now let $X$ be a metric space and consider a subset $S \subset X$.  Recall that the \it $d$-dimensional Hausdorff content \rm of $S$ is defined by
$$
C^d_H(S) := \inf \left\{ \sum_i r_i^d : \text{ there is a cover of $S$ by balls of radii } r_i> 0 \right\},
$$
and the \emph{Hausdorff dimension} of $X$ is 
$\dim_H(X) :=\inf \{ d \geq 0 : C_H^d(X) = 0 \}.$

We then have the following Proposition.

\begin{pro} If $\Gamma $ is a group, $\dim_H (\hat\Gamma) = -\liminf_{n \to \infty} \frac{\log  i^ {\lhd}_\Gamma(n)}{\log\rho (n)}$.
\end{pro}

For instance, if $\rho(n) = e ^ {- n}$ then the Hausdorff dimension is the coefficient $c$ in exponential intersection growth $e^{c n}$.  If $\rho (n) =\frac 1n$, then $\dim_H (\hat\Gamma)$ is the degree of polynomial intersection growth.

There is actually no difference in Hausdorff dimension between the profinite completion $\hat\Gamma $ and the group $\Gamma $, considered with the (pseudo)-metric $d_\rho $.  A priori, the Hausdorff dimension of $\Gamma $ could be less, but the second half of the proof below works just as well for $\Gamma $ as for $\hat\Gamma $.

For simplicity, we stated this proposition for normal intersection growth.  However, after changing the definition of $d_\rho$ by considering only subgroups $\Delta $ in a class $\bullet$, an analogous result follows for $\bullet $-intersection growth.

\begin{proof}
The main point here is that $B_{d_\rho} \left(e,\rho (n)\right)= \overline {\Lambda^\lhd_\Gamma (n) }, $  where the set on the left is the  ball of radius $\rho (n) $ around the origin in $\hat\Gamma $.  To prove that 
$$\dim_H (\hat\Gamma) \geq -\liminf_{n \to \infty}\frac{\log  i ^ {\lhd}_\Gamma(n)}{\log\rho (n)}, $$
we just note that if $d $ is greater than the right-hand side then there are arbitrarily large $n $ such that $$d \geq - \frac {\log i ^ {\lhd}_\Gamma (n)} {\log\rho (n)}. $$  However this implies that ${i ^ {\lhd}_{\Gamma} (n)} \leq \rho (n) ^{-d }$, so using cosets of $\overline {\Lambda^\lhd_{\Gamma} (n)} $ we can cover $\hat\Gamma $ with at most $\rho (n) ^ {-d} $ balls of radius $\rho (n) $.  Therefore, we have that the $d $-dimensional Hausdorff content of $\hat\Gamma $ is at most $1 $.  This proves the first half of the proposition.

We must now show that $$\dim_H (\hat\Gamma) \leq -\liminf_{n \to \infty}\frac{\log  i ^ {\lhd}_\Gamma(n)}{\log\rho (n)}. $$
That is, if $d$ is less than the right-hand side then we need to show that the $d$-Hausdorff content of $\hat\Gamma $ is greater than zero.  Suppose that $\{B_i\} $ are balls with radii $\rho (n_i) $ that cover $\hat\Gamma $.  Choosing the radii to be small, we may assume that all $n_i $ are large enough that $d < -\frac {\log i ^ {\lhd}_{\Gamma} (n)} {\log\rho (n)} $.   If $\mu $ is the Haar (probability) measure on $\hat\Gamma $, then 
\begin {align*}
1 \leq \sum_i\mu (B_i)
=\sum_i \frac 1 {i ^ {\lhd}_{\Gamma} (n_i)}
\leq \sum_i \rho (n_i) ^d,
\end {align*}
which shows that the $d $-dimensional Hausdorff content of $\hat\Gamma $ is at least $1 $.
\end{proof}

\section{Intersection growth of polycyclic groups and $\SL_k(\BZ)$
\label{sec:intgrowth-Z-SL}}

Of course, the investigation of any new growth function should begin with the following example.

\begin{pro}
\label{Zcase}
For the group $\BZ ^ k $, we have
$
i_{\BZ ^ k} (n) = %i^{\lhd}_{\BZ ^ k} (n) =
\lcm \{1,\ldots, n\} ^ k  ,$
and $
i^{\max}_{\BZ ^ k} (n)= %i^{\max\lhd}_{\BZ ^ k} (n) =
\lcm\{p | p< n, \mbox{$p$-prime}\} ^ k.
$  
\end{pro}

\begin{proof}
Using Proposition~\ref{pro:products}, it is enough to verify the statement for $k=1$.
Since $\BZ$ has a unique (automatically normal) subgroup of index $l$ for each $l$,
$$
\Lambda_{\BZ} (n) = \bigcap_{l \leq n} l\BZ = \lcm \{1,\ldots, n\}\BZ.
$$
Now Corollary \ref {lcm} yields that $i_{\BZ}(n) = \lcm (1,\ldots, n) \asymptotic e^n$.

The situation with maximal subgroups is similar -- the subgroup $l\BZ$ is maximal in $\BZ$ only when $l$ is prime, so
$$
\Lambda^{\max}_{\BZ} (n) = \bigcap_{p \leq  n} p\BZ = \left(\prod_{p \leq n} p\right) \BZ,
$$
and $i^{\max}_{\BZ}(n) = \left(\prod_{p \leq n} p\right) \asymptotic e^n$.

One can also give a profinite version of this argument, as $i_{\BZ}^\bullet=i_{\widehat{\BZ}}^\bullet$ by Lemma \ref{pro-finite}.
The chinese reminder theorem gives that $\widehat{\BZ} = \prod_p \BZ_p$, so one only needs to compute the functions $i^{\bullet}_{\BZ_p}$ and apply Proposition \ref{pro:products}.\end{proof}

Combined with Proposition \ref {pro:products} and the 
prime number theorem (see Corollary \ref {lcm}), Proposition 5.1 gives

\begin {cor}\label {fga}
If $\Gamma $ is a finitely generated abelian group, then $i_\Gamma ^\bullet(n) \asymptotic e^n$, where $\bullet $ is $\max$, $\max \lhd$, $\lhd$ or no symbol.  Moreover, if $\Gamma $ has rank $k $, then
$$\log i_\Gamma^{\bullet}(n)  \congruent  kn.$$
\end{cor}
As in section \ref {basics}, $f \congruentleq g$ means that the $\limsup_{n\to\infty} f(n) /g(n)\leq 1 $.  Most of our calculations of $\log i_\Gamma(n)$ only work up to multiplicative error, but in the beginning of this section some finer calculations are possible.

A \emph {polycyclic group} $\Gamma $ is a group that admits a subnormal series $$\Gamma = \Gamma_1 \rhd \Gamma_2 \rhd \cdots \rhd \Gamma_k = \{e\}$$ in which all the quotients $\Gamma_i / \Gamma_{i+1}$ are cyclic.  Examples include finitely generated nilpotent groups and extensions of such groups by finitely generated abelian groups.  The number of infinite factors $\Gamma_i / \Gamma_{i+1}$ in such a subnormal series is called the \emph {Hirsch length} of $\Gamma$.

\begin {pro} \label {polycyclic}
Any infinite polycyclic group $\Gamma $ has $ i_\Gamma(n) \asymptotic e^n$.  Specifically, we have that $\frac nc \congruentleq \log i_\Gamma(n) \congruentleq kn,$ where  $c$ is the smallest index of a subgroup of $\Gamma$ with infinite abelianization and $k$ is the Hirsch length of $\Gamma$.
\end {pro}

\begin {proof}
Suppose that $\Delta $ is an index $c$ subgroup of $\Gamma$ with infinite abelianization.  Then by Lemma \ref {finiteindex}, $i_\Gamma (n) \leq c \cdot \intfunc_\Delta (n) \leq  \intfunc_\Gamma(c n)$. Observation
\ref{thm:quotients} now yields the lower bound, since
$\log i_\Delta(n) \geq \log i_{\Delta^{ab}}(n) \congruentgeq n$.   For the upper bound, let $N < \Delta $ be a subgroup with $\Delta / N$ infinite cyclic.  Then as $N$ has Hirsch length $k-1$, we have by induction, Proposition \ref {pro:quotients} and Corollary \ref {fga} that $$\log i_\Gamma(n) \congruentleq \log i_\Delta(n) \congruentleq n + (k-1)n=kn.\qedhere$$
\end {proof}

Proposition \ref {polycyclic} certainly gives upper bounds for the normal, maximal normal and maximal intersection growth of polycyclic groups.  However, we remind the reader that as in Example \ref {qex} these upper bounds may not be sharp.  In fact, the same example illustrates how the lower bound for $i_{\Gamma} (n) $ may be affected by the index of a subgroup with infinite abelianization.
%\begin{lem}%\label{obvious}
%Let $\Gamma$ be a group. Then $i^{max}_{\Gamma}(n)  \le i^{\lhd}_{\Gamma}(n)
%\le i_{\Gamma}(n)$.
%\end{lem}}

%\marginpar{I know that both $\log i_{\BZ}(n)$ and $i^{\max}_{\BZ}(n)$ are $\congruent n$ in our notation. However if we exponentiate this is no longer true.}

\vspace{2mm}

Here is a first calculation of intersection growth in a non-polycyclic group. %We analyze the normal subgroup intersection growth first.

\begin{thm} \label{SL_3}
For the special linear groups $\SL_k (\Z)$, where $k\geq 3$, we have 
$$  i_{\SL_k(\Z)}^\lhd(n) \asymptotic i_{\SL_k(\Z)}^{\max \lhd}(n) \asymptotic
 e ^{n^{1/(k^2-1)}}, \ \ \text{ but } \ \ 
i_{\SL_k(\Z)}^{\max}(n) \asymptotic  e ^{n^{1/(k-1)}}.$$
\end{thm}
%\marginpar{TODO: We can even do all subgroups, but I do not want to do it today.}

We believe this result extends to other split higher rank Chevalley groups, where the numbers $k^2-1$  and $k-1$
should be replaced with the dimension of the group and the dimension of the smallest projective variety on which the group acts faithfully.
We added notes to the parts of our proof that would have to be modified to obtain such a generalization.
%The exponent $k^2 - 1 $ is the dimension of $\SL_k(\BR) $.
%We expect that the obvious generalization holds for split higher rank Chevalley groups.
%\marginpar{We actually have this!!!}

\begin{proof}
One of the main ingredients in the proof is the congruence subgroup property. One way to state this is that the map
$$
\pi : \widehat{\SL_k(\Z)} \to \SL_k(\widehat{\Z})
$$
is an isomorphism.%
\footnote{For some split higher rank Chevalley groups the map is not an isomorphism, but its kernel is finite and central (see \cite{BMS67}). 
This does not significantly affect the following estimates.}
Here $\widehat{\SL_k(\Z)}$ denotes the profinite completion of the group $\SL_k(\Z)$ and $\widehat{\Z}$ is the profinite completion of the ring $\Z$.
Since $\widehat{\Z} = \prod_p \Z_p$, we then have a product decomposition
$$
\widehat{\SL_k(\Z)} \simeq \prod_p \SL_k(\Z_p).
$$
So, Proposition~\ref{pro:products} reduces the computation of $i^{\bullet}_{\SL_k(\Z)}$ to estimates of $i^{\bullet}_{\SL_k(\Z_p)}$.

Let $H$ be a normal subgroup of finite index in $\SL_k(\Z_p)$. Then there exists minimal integer $s$ called the \emph{level} of $H$ such that
$H$ contains the congruence subgroup  $\SL_k^s(\Z_p) = \ker ( \SL_k(\Z_p) \to \SL_k(\Z_p/p^s\Z_p))$.

\begin{lem}
\label{lm:levelsubgroups}
If $H$ is a normal subgroup of $\SL_k(\Z_p)$ of level $s$, 
then the image of $H$ in $\SL_k(\Z_p/p^s\Z_p)$ is central.
\end{lem}
\begin{proof}
We prove that a noncentral normal subgroup $H $ of $\SL_k(\Z_p/p^s\Z_p)$ must contain $\SL_k^ {s - 1}(\Z_p/p^s\Z_p)$.  First, assume that $H \cap \SL_k^1(\Z_p/p^s\Z_p)$ contains a noncentral element $g$.  We can then write $$g=Z+p^iA, \ \ 1\leq i \leq s-1,$$ where $Z$ is central and congruent to the identity mod $p$, while $A$ is noncentral modulo $p$.  The binomial theorem\footnote{If $p\neq 2,$ then $(Z+p^iA)^p = Z^p + p^{i+1}A'$, where $A'\equiv AZ^{p-1}$ is noncentral modulo $p$, so we can increase $i $ by taking powers.  If $p=2$ and $i=1$, we have $A'\equiv AZ+A^2$,which may be central modulo $p$.  However, using the irreducibility of $\SL_k(\Z_2) $ acting on $\ssl_k(\mathbb{F}_2)$ modulo its center, one may replace $g\in H$ by another element $Z+2^iA\in H$ where this is not the case.}
 then shows that after replacing $g$ by an iterated $p^{th}$ power we may assume that $i =s-1$.  The only invariant $\SL_k(\Z_p)$-submodule of $\ssl_k(\mathbb{F}_p)$ is its center, so as $A$ above is noncentral and is naturally an element of $\ssl_k(\mathbb{F}_p)$, it follows that by multiplying together sufficiently many conjugates of $g$ we can generate any element of the form $1+p^{s-1}A'$, so $H$ contains $\SL_k^ {s - 1}(\Z_p/p^s\Z_p)$.

If $H \cap \SL_k^1(\Z_p/p^s\Z_p)$ is central, every element of $H$ is central modulo $p^{s-1}$.  As $H $ is noncentral, it then has an element of the form $Z+p^{s-1}A$, where $A$ is noncentral modulo $p$, and the argument finishes as above.
\end{proof}

\comment{%%%%%%%%%%%%%%

begin{lem}
\label{lm:levelsubgroups2}
If $H \leq \SL_k(\Z_p)$ has level $s$, then $[\SL_k(\Z_p):H] \geq (p/2)^{s(k-1)}.$
\end{lem}
\begin {proof}
Let $H $ be a subgroup of $\SL_k(\Z/p^s\Z)$ and suppose that the intersection $I=H \cap \SL_k^ {s - 1}(\Z/p^s\Z)$ is not all of $ \SL_k^ {s - 1}(\Z/p^s\Z) \cong \ssl_k(F_p)$.  Then $I $ is identified with an index 
\end {proof}}%%%%%%%%%%%%%%

This lemma says that, up to a small perturbation coming from the centers, the congruence subgroups $\SL_k^s(\Z_p)$ are the only normal subgroups of $\SL_k(\Z_p)$. As the size of the center of $\SL_k(\Z_p/p^s\Z_p)$ is always less than $k$, this implies that $i^{\lhd}_{\SL_k(\Z_p)}(n)$ is approximately linear.  More precisely,
\begin{pro}
We always have
$$
\frac{1}{p^{k^2-1}}\leq \frac{i^{\lhd}_{\SL_k(\Z_p)}(n)}{n} \leq k
$$
and
$$
i^{\max \lhd}_{\SL_k(\Z_p)}(n)= \left\{
\begin{array}{cl}
1 & \mbox{ if  }n <  |\PSL_k(\F_p)| \\
|\PSL_k(\F_p)| & \mbox{ if  } n \geq  |\PSL_k(\F_p)|,
\end{array}
\right.
$$
$$
i^{\max}_{\SL_k(\Z_p)}(n)= \left\{
\begin{array}{cl}
1 & \mbox{ if  }n <  |\F_pP^{n-1}| \\
|\PSL_k(\F_p)| & \mbox{ if  } n \geq  |\F_pP^{n-1}|.
\end{array}
\right.
$$
where $\F_pP^{n-1}$ is the projective space of dimension $n-1$ over $\F_p$.
\iffalse
\begin{enumerate}
\item if $n < |\PSL_k(\F_p)|$ then $i^{\lhd}_{\SL_k(\Z_p)}(n)=i^{\max \lhd}_{\SL_k(\Z_p)}(n) = 1$;
\item if $n \geq |\PSL_k(\F_p)|$ then $i^{\max \lhd}_{\SL_k(\Z_p)}(n) = |\PSL_k(\F_p)|$;
\item if $n \geq |\PSL_k(\F_p)| p^{(k^2-1)s}$ then $i^{\lhd}_{\SL_k(\Z_p)}(n) \geq |\PSL_k(\F_p)| p^{(k^2-1)s}$;
\item if $p\not | k$ and $n <|\PSL_k(\F_p)| p^{(k^2-1)s}/k$ then $i^{\lhd}_{\SL_k(\Z_p)}(n) < |\PSL_k(\F_p)| p^{(k^2-1)s}$;
\item if $p | k$ and $n <|\PSL_k(\F_p)| p^{(k^2-1)s}/p$ then $i^{\lhd}_{\SL_k(\Z_p)}(n) < |\PSL_k(\F_p)| p^{(k^2-1)s}$.
\end{enumerate}
\fi
\end{pro}
\begin{proof}
The upper bound $\frac{i^{\lhd}_{\SL_k(\Z_p)}(n)}{n}$ follows from Lemma~\ref{lm:levelsubgroups} and the observation above that the size of the center of
$\SL_k(\Z_p/p^s\Z_p)$ is less than $k$.  The lower bound comes from the ratio of the size of $\SL_k(\Z_p/p^s\Z_p)$ and $\SL_k(\Z_p/p^{s+1}\Z_p)$.  The formula for the maximal normal intersection growth of  $\SL_k(\Z_p)$ is restating the fact that this group has a unique maximal normal subgroup,
%\marginpar{Francesco: I corrected the part after the comma, is what's written correct? MK:added `is`}
while the last formula uses that the smallest (with respect to the number of points) nontrivial action of
$\PSL_k(\F_p)$ is on the projective space $\F_pP^{n-1}$.
\end{proof}

%This lemma say that $i^{\lhd}_{\SL_k(\Z_p)} \asymptotic n$ for any fixed $k$ and $p$. However $i^{\max \lhd}_{\SL_k(\Z_p)}(n) =1$ if $n < |\PSL_k(\F_p)| \sim p^{k^2-1}$.

Now we are ready to estimate $i^{\lhd}_{\SL_k(\Z)}$:
$$
i^{\lhd}_{\SL_k(\Z)}(n) = \prod_p i^{\lhd}_{\SL_k(\Z_p)}(n) = \prod_{p\leq 2\sqrt[k^2-1]{n} } i^{\lhd}_{\SL_k(\Z_p)}(n)
$$
The second equality follows from Lemma \ref{lm:levelsubgroups}, which implies that $\SL_k(\Z_p)$ has no normal subgroups of index less than $ {|\PSL_k(\F_p)|} \geq {(p/2)^{k^2-1}}$.   Then
$$
i^{\lhd}_{\SL_k(\Z)}(n) \ \leq  \prod_{p\leq 2\sqrt[k^2-1]{n}} kn \ \asymptoticless \  e^{(k^2-1)(2\sqrt[k^2-1]{n} )},
$$
%\marginpar{Francesco: sentence is not very clear. One need to check that formula below and remove $p>k$, probably, MK: is OK now?}
%for $p>k$ 
where we estimate the products over primes using the prime number theorem.
This inequality implies
$i^{\lhd}_{\SL_k(\Z)}(n) \asymptoticless e ^{\sqrt[k^2-1]n}$.

Similarly, we can obtain a lower bound for $i_{\SL_k(\Z)}^{\max \lhd}(n)$:
\begin{align*}
i^{\max \lhd} _{\SL_k(\Z)}(n) \ &= \ \prod_p i^{\max \lhd}_{\SL_k(\Z_p)}(n) \geq \prod_{p \leq \sqrt[k^2-1]{n} }  i^{\max \lhd}_{\SL_k(\Z_p)}(n) \\
&\geq \prod_{p\leq \sqrt[k^2-1]{n}} (p/2)^{k^2-1} \  \asymptoticge \  e^{\sqrt[k^2-1]{n}},
\end{align*}
where the justification of the inequalities is the same as above, except that for the last one we appeal to Lemma \ref {primepowers}.

Therefore we have inequalities
$$
e ^{\sqrt[k^2-1]n}
\succeq
i^{\lhd}_{\SL_k(\Z)}
\succeq
i^{\max \lhd}_{\SL_k(\Z)}
\succeq
e ^{\sqrt[k^2-1]n}.
$$
The computation for  $i^{\max}_{\SL_k}$ is essentially the same but the products are only over $p \leq \sqrt[k-1]{n}$ which is, loosely speaking, almost equivalent to
$|\F_pP^{k-1}| \leq n$.
\end{proof}

%This result can be generalized in a straightforward way to Chevalley groups of rank $\ge 2$ and a variation of the  argument can be made for normal intersection growth. Details will appear in a forthcoming article.

\section {Intersection growth of nonabelian free groups}

By Observation \ref{thm:quotients}, the free group of rank $k$ has the fastest-growing intersection growth functions among groups generated by $k$ elements.  Here are their asymptotics:

\begin{thm}
\label{maximal}
Let $\FF ^ k $ be the rank $k $ free group. Then
$$
i_{\FF ^ k}^ {\max\lhd} (n) \asymptotic  e ^ {\left( {n ^ { k -2/3}} \right )}, 
\quad
 i_{\FF ^ k}^ {\max} (n) \asymptotic  i_{\FF ^ k} (n)  \asymptotic e ^ {\left( {n ^ n} \right )}.
$$
\end{thm}

The proof for $i_{\FF ^ k}^ {\max\lhd} (n)$ uses the classification theorem for finite simple groups: we calculate separately the index of the intersection of subgroups with quotient a fixed finite simple group and then combine these estimates to give the asymptotics above.  In fact, we will show that the growth rate of
$i_{\FF ^ k}^{\max\lhd}(n) $ is dictated by subgroups with quotient $\PSL_2 (p) $.  The contributions of other families of finite simple groups are comparatively negligible.

The lower bound for $i_{\FF ^ k}^ {\max} (n) $ and $  i_{\FF ^ k} (n) $ comes from the alternating group $A(n)$.  Since the (maximal) index $n $ subgroups of $A(n)$ intersect trivially, one automatically gets a factorial lower bound for intersection growth from any surjection $\FF^k \longrightarrow A(n)$.   We will see that multiplying the estimates that one gets from all possible surjections gives the lower bound of $e ^ { {n ^ n}} $.  As this requires some of the same machinery as does the calculation of $i_{\FF ^ k}^ {\max\lhd} (n)$, we will finish this argument at the end of the section.

However, the proof of the upper bound of $e ^ {\left( {n ^ n} \right )}$ is completely general:

\begin {pro}\label {universalbounds}
If $\Gamma$ is a finitely generated group, then $i_{\Gamma} (n) \asymptoticless e^{n^n}$.
\end {pro}
\begin {proof}
 If $H \leq \Gamma$ is a subgroup with index $i$, then $H$ contains the kernel of the map $\Gamma \longrightarrow S_i$ determined by the action of $\Gamma $ on the cosets of $H $.  Thus,
$$\bigcap_{\substack {H \leq \Gamma, [\Gamma : H] = i} }H \ \ \supseteq \ \ \bigcap_{f: \Gamma \to S_i} \ker f.$$
Each such kernel has index at most $i!$, and if $\Gamma $ is $k $-generated, there are at most $(i!)^k$ homomorphisms from $\Gamma $ to $S^i$. So, this implies that
$$i_{\Gamma}(n) \leq \prod_{i=1}^n (i!)^{(i!)^k} \asymptotic  \prod_{i=1}^n (i^i)^{(i^i)} \asymptotic  \prod_{i=1}^n (i^{i^i}) \asymptotic \prod_{i=1}^n (e^{i^i}) \asymptotic e^{n^n}.\qedhere$$
\end {proof}

We now start on the calculation of the asymptotics of $i_{\FF ^ k}^ {\max\lhd} (n)$.  The key is to consider first the index of the intersection of subgroups with a specific quotient and then to analyze how these estimates combine.  As abelian quotients are easy to handle, we focus mostly on the non-abelian case.

Fix a finite simple group $S $ and let $i_{\FF ^ k} (S)$ be the index of the subgroup
$$
\Lambda_{\FF ^ k} (S) =\bigcap_{\substack{ \Delta \lhd \FF ^ k \\ \FF ^ k / \Delta \cong S } } \Delta \ \ \leq \ \FF^k.
$$

%\begin {proof}
%The group $G/\Delta_1\Delta_2$ is simultaneously a quotient of $G/\Delta_1$ and
%$G/\Delta_2$ and it would be a nontrivial isomorphic quotient unless $G=\Delta_1\Delta_2$.
%Then the map $\tau: G \longrightarrow G/\Delta_1 \times G/\Delta_2 $ is a
%surjection with kernel $\Delta_1\cap \Delta_2$.
%\end{proof}

\begin{pro}
\label {whatisd}\label {localfactors}
If $S $ is a non-abelian finite simple group, then $\FF ^ k / \Lambda_{\FF ^ k} (S) \cong  S  ^ {d(k,s)} $,
where $d(k,S)$ is the number of $\Delta \lhd \FF ^ k $ with $\FF ^ k/\Delta \cong S $. \end{pro}
%Moreover, the growth function $i_{\FF ^ k} ^ {\max\lhd}(n)$ factors as
%$$
%i_{\FF ^ k} ^ {\max\lhd}(n) = \prod_{\substack {\text {finite simple groups} \\ S \text { with } | S | \leq n}} i_{\FF ^ k} (S).
%$$

\begin{proof}
If $\Delta_1,\ldots,\Delta_i \trianglelefteq \FF ^ k  $ are distinct normal subgroups with $\FF ^ k  /\Delta_i \cong S $, 
\begin {equation*}\label {productequation}
\begin{array}{cccc}
\varphi: & \FF ^ k / \left(\Delta_1 \cap \ldots \cap \Delta_i\right ) & \hookrightarrow  & \FF ^ k  / \Delta_1 \, \times \, \cdots \, \times \, \FF ^ k / \Delta_i, \\
& g\left(\Delta_1 \cap \ldots \cap \Delta_i\right)  & \mapsto & (g\Delta_1, \ldots, g\Delta_i).
\end{array}
\end {equation*}
is the required isomorphism.
\end{proof}

The advantage of Proposition~\ref{localfactors} is that $d (k, S) $ is easily computed: namely, observe that $d (k, S) $ measures exactly the number of generating $k $-tuples in $S $, modulo the action of the automorphism group $\automorphism (S) $.  In other words,

\begin{lem}\label {bounded}
If $S $ is a non-abelian finite simple group, then
$$
 d (k, S) = p(k,S) \frac{\displaystyle | S | ^ {k}}{\displaystyle |\Aut(S)|} \ \congruent \ \frac {| S | ^ {k- 1}} {|\Out (S) |} .
$$
where $p(k,S)$ is the probability that a $k$-tuple of elements in $S$ generates.
\end{lem}
\begin{proof}
For the first equality, note that under the action of $\Aut(S)$ on $S^k$ each generating tuple has orbit of size $ |\Aut(S)|$.  Since $S $ is non-abelian and simple, the conjugation action of $S $ is faithful, so $|\automorphism (S) | = | S | |\Out (S) | $.  Finally, Liebeck-Shalev~\cite{Shalev} and Kantor-Lubotzky~\cite{Lubotzky} have shown that for any fixed $k$ the probability $p(k,s)$ tends to $1$ as $|S| \to \infty$, which gives the asymptotic estimate.
\end{proof}

We mention that recently Menezes,  Quick and Roney-Dougal 
\cite[Theorem 1.3]{MQR13} have given an explicit lower bound for $d(k,S)$
and show when it can be attained.

%that is $d(k,S) \geq \alpha \frac{|S|^{k-1}}{\log|S|}$ for a suitable explicit $\alpha$
%and where equality is attained when $k=2$ and $S=\mathrm{L}_3(4)$.

%Recently (student of Colva Roney-Dugal) have shown that $p(k,S) \geq p(2, A_5) = ??$.
%\marginpar{can we quickly find a reference?, if not just rephrase}

\iffalse
%Also, for each $m\in \BN $ there is some $\epsilon(m) >0 $ such that for all primes $p $:
%$$ \epsilon(m)  | \PSL_m (p) | ^ {k - 1} \leq d (k,\PSL_m (p )) \leq | \PSL_m (p) | ^ {k - 1}. $$
%\end{pro}
\begin {proof}
Liebeck-Shalev~\cite{Shalev} and Kantor-Lubotzky~\cite{Lubotzky} have shown that there is some $\epsilon' >0 $ such that if $S $ is a finite simple group, then
$$
\epsilon' | S | ^ k \leq |\{\text { generating $k $-tuples in } S\ \}| \leq {| S | ^ k} .
$$
\marginpar{a student of Colva Roney-Dugal had a explicit bound for $\epsilon'$ -- just look at number of generating tuples of $A_5$}

Since $S $ has no center, the action of $\automorphism (S) $ on the space of generating $k $-tuples is free.  From this, we see that
$$
 \frac {\epsilon'| S | ^ k} {|\automorphism (S) |} \leq d (k, S) \leq \frac {| S | ^ k} {|\automorphism (S) |}.
$$
Now to prove the proposition we must only give bounds for $|\automorphism (S) | $.  Inner automorphisms show that it is always at least $| S | $, which proves the first half of the proposition.  For the second half, one needs only recall (see~\cite{Wilson}) that
$|\outerautomorphism (\PSL_m (p)) | = 2 (m, q - 1)$ for $m\geq 3 $ and $|\outerautomorphism (\PSL_2 (p))| = (2, p) $.
\end{proof}
\fi

\begin {cor}\label {fixedgroup}
There is some fixed $\epsilon >0 $ such that if $S $ is a finite non-abelian simple group, then we have the estimate
$$
 | S | ^ {\, \epsilon\cdot \frac {|S| ^ {k - 1}} {|\Out (S) |} } \leq i_{\FF ^ k}(S) \leq | S | ^ {\frac {|S| ^ {k - 1}} {|\Out (S) |} }.
$$
\end{cor}

\vspace{3mm}

To calculate $i_{\FF ^ k}^ {\max\lhd} (n)$, we now analyze the intersections of subgroups of $\FF ^ k $ with quotients lying in a given infinite family of finite simple groups.  Table~\ref {groups} gives the classification of infinite families; the list includes all finite simple groups other than the finitely many `sporadic' groups.

%%%%%%%%%%%%%%%%%%%%%%%%%%%%%%%%%%%%%%%%%%%

\begin {table}[htbp]
\newcolumntype {Q}{ >{$}c <{$}}
\begin {tabular}{|Q | Q | Q |}
\hline
\mbox{Family}   &  \mbox{Approximate Order} & \mbox {Approximate $|\Out | $} \\
\hline
 \BZ/p\BZ, \, p \mbox{ prime }   &  p & p - 1 \\
 A (m), \,m \geq 5    &  \frac{m!}{2} & 2, \text { unless } m = 6 \\
 \PSL_m (q), \,m \geq 2   &   \frac{1}{(m, q- 1)} q^{m ^ 2 - 1}&  (m,q - 1)\cdot s\\
 B_m (q), \, m \geq 2     &   \frac{1}{(2, q - 1)} q^{2m^2+m}  &  s\\
 C_m (q), \, m \geq 3     &   \frac{1}{(2, q - 1)} q^{2m^2+m}& s \\
 D_m (q), \, m \geq 4     &   \frac{1}{(4, q m - 1)}q^{2m^2 -m}& s \\
 ^ 2 A_m (q^ 2), \, m \geq 2     &  \frac {1} {(m + 1, q - 1)} q ^ {m^2 +2m +1 }& (m + 1,q + 1)\cdot s \\
 \  ^ 2 D_m (q^ 2), \, m \geq 4 \      &  \frac {1} {(4, q^ m + 1)} q ^ {2m^2 - m} & s \\
 E_6 (q)    &   \frac{1}{(3, q - 1)} q ^ {78} & s  \\
 E_7 (q)   &   \frac{1}{(2, q - 1)} q ^ {133} & s\\
 E_8 (q)   &   q ^ {248} & s\\
 F_4 (q)   &   q ^ {52}  & s\\
 G_2 (q)   &   q ^ {14} &  s\\
 ^2E_6 (q^ 2)    &  \frac{1}{(3, q + 1)} q ^ {78}  & s\\
 ^3 D_4 (q ^ 3)   &  q ^ {28} & s \\
 ^2 B_2 (2 ^ {2 j + 1})    &  q^5,  \mbox{where } q = 2 ^ {2 j + 1} & 2 j + 1  \\
 ^2 G_2 (3 ^ {2 j + 1})    &  q^7 , \mbox{where } q = 3 ^ {2 j + 1} & 2 j + 1 \\
 ^2 F_4 (2 ^ {2 j + 1})    &   \  q^{26}, \mbox{where } q = 2 ^ {2 j + 1} \  & 2 j + 1  \\
\hline
\end {tabular}
\caption{\footnotesize Infinite families of finite simple groups, their sizes and the sizes of their outer automorphism groups.  We assume  $m, j\in \BN$   and that  $ q = p ^ s$  is  a prime power.  The approximations given are true up to a universal multiplicative error.}
\label{groups}
\end {table}
Let $i_{\FF ^ k}^{\CG}(n)$ be the index of the intersection of all normal subgroups of $\FF ^ k $ with index at most $n $ and quotient lying in a family $\CG $ of finite simple groups.  It will be convenient to split the rows in Table \ref {groups} into \it single parameter \rm families: if a row has two indices, we fix $m $ while varying $q $.  Examples of single parameter $\CG $ include $A $, $^ 2 E_6 $, $\PSL_2 $, $\PSL_3 $, etc.

\begin {pro}  There is a product formula
\label {families}
$$i_{\FF_k} ^ {\max\lhd} (n) \ \ 
\asymptotic \prod_{\substack { \text {single parameter} \\ \text {families } \CG }}
i_{\FF_k} ^ {\CG} (n) \ \ = \ \  \prod_{\CG } \, \prod_{\substack {S \in \CG \\ | S | \leq n}} i_{\FF  ^ k} (S).$$
\end {pro}

The multiplicative discrepancy is due to the absence of the sporadic groups in the product: since there are only finitely many of them they contribute at most a multiplicative constant to $i_{\FF ^ k} $.  Also, although the product is infinite, for each $n $ there are at most $Cn $ non-unit factors for some universal  $C $.

\begin{proof}[Proof of Proposition \ref {families}]
This follows inductively from the fact that if $\Delta_1,\Delta_2$ are normal subgroups of $\FF  ^ k $ such that $G/\Delta_1 $ and $G/\Delta_2 $ have no nontrivial isomorphic quotients, then
$
\FF  ^ k/(\Delta_1\cap\Delta_2)\cong \FF  ^ k/\Delta_1 \times \FF  ^ k/\Delta_2. \qedhere
$
\end{proof}

\iffalse %%%%%%%%%%%%OLD%%%%%%%%%%%%%%
\begin{pro}
\label{PSL}
For any single parameter family% 
\footnote{The cases of Suzuki and Ree groups $^2B_2$, $^2G_2$ and $^2F_4$ are slightly different since $q$ can only be a power of fixed prime, this leads to replacing the term $n^{\frac{1}{d}}$ with $\log n$.}
of Lie type $\CG$ then
$$
\log i_{\FF ^ k} ^{\CG,\lhd} (n)  \asymptotic n ^ {(k - 1) + \frac{1}{d}}
\quad
\mbox{and}
\quad
\log i_{\FF ^ k} ^{\CG} (n)  \asymptotic  n ^ {\frac{d(k - 1) +1}{d'}},
$$
where $d$ is the dimension of the corresponding Lie group and
$d'$ is the smallest dimension of a variety it can act on. More precisely
$
\log i_{\FF ^ k} ^{\CG} (n) <2^{2dk}  n ^ {(k - 1) + \frac{1}{d} }.
$
\fi%%%%%%%%%%%%%%%%%%%%%%%%%%%

\begin{pro}
\label{PSL}
For a single parameter family $\CG $ of finite simple groups,
\begin {enumerate}
\item $\log i_{\FF ^ k} ^{\CG} (n)  \asymptotic n $, if $\CG$ is the family of cyclic groups.
\item $\log i_{\FF ^ k} ^{\CG} (n)  \asymptotic n^{k-1} $, if $\CG=^2B_2, ^2G_2,$ or $^2F_4$.
\item $\log i_{\FF ^ k} ^{\CG} (n)  \asymptoticless n^{k-1}\log(n) $, if $\CG =A.$ 
\item $\log i_{\FF ^ k} ^{\CG} (n)  \asymptotic n ^ {(k - 1) +\frac{1}{d}},$ if $\CG $ is one of the remaining families of Lie type and $d $ is the dimension of the corresponding Lie group, which appears in Table \ref {groups} as the exponent of $q$.  More specifically, there is some universal $C $ such that
$
\log i_{\FF ^ k} ^{\CG} (n) \leq C \cdot d^ {k +\frac 1d} \cdot n ^ {(k - 1) +\frac{1}{d}}.
$
\end {enumerate}
\end {pro}

As the dimension $d$ is uniquely minimized when $\CG = \PSL_2$, this implies that the family $\PSL_2 $ has the fastest intersection growth.  The point of Theorem \ref{maximal} is that this is faster than the growth of all other families combined.

\begin {proof}
The cyclic case is essentially Corollary \ref {fga}. For all the others, we will combine Corollary \ref {fixedgroup} with the product formula in Proposition \ref {families}.  

When $\CG=\,^2B_2,$ we have that
\begin{align*}
\log i^{\CG}_{\FF ^ k}(n) \ \ & = \ \ \sum_{q\leq n} \ \ \log \left (|{^2B_2} (q) | ^ {\frac {|{^2B_2}(q) | ^ {k - 1}}{|\Out |}}  \right)\\
& = \sum_{2 ^ {10 j + 5} \leq n} \frac {(2 ^ {10 j + 5}) ^ {k - 1}} {2 j + 1} \log\left (2 ^ {10 j + 5}\right)\\
& \asymptotic  n ^ {k - 1},
\end {align*}
where the last step comes from the fact that a sum of exponentially increasing terms is proportional to the last term.  The Ree groups $^2G_2$ and $^2F_4$ admit similar computations.  For the alternating group,
\begin{align*}
\log i^{A}_{\FF ^ k}(n)  \  \asymptotic  \sum_{m!/2\leq n} \ \ {\frac {(m!/2) ^ {k - 1}}{2}}  \log ( m!/2)  \ \asymptoticless \ n ^ {k - 1} \log (n).
\end {align*}
Again the sum is proportional to its last term, but if $n= \frac {m!} 2 - 1 $, this last term is much smaller than $n ^ {k - 1} \log (n)$, and we stop with the upper bound.  The difference between this and the computation for Suzuki and Ree groups is that the gaps between successive factorials are large enough to make an asymptotic estimate that works for all $n $ unwieldy.

For the last case, we use the notation $\leq_c, \, \geq_c, \, =_c$ for comparisons that are true up to a \emph {universal} multiplicative error, in contrast to those in the statement of the proposition, where the error may depend on the family $\CG $.  

First, note that from Table \ref {groups}, for any of the groups $\CG (q) $ in $(4) $ we have $$\frac 1{d} q^{d} \leq_c\  |\CG(q)| \leq_c \ q ^ {d},$$ where the $\frac 1d$ is to account for the gcd's in $|\PSL_m(q) |$ and $|\, ^ 2A_m(q ^ 2) |$.  So for \emph {fixed} $\CG$ we have $|\CG(q)| \asymptotic q^d$. By Corollary \ref {bounded},
\begin{align*}
\log i_{\FF ^ k}(\CG (q)) & \le \log \left (|\CG (q) | ^ {|\CG(q) | ^ {k - 1}}  \right)\\ &= |\CG (q) | ^ {k - 1} \log |\CG (q) | \\
& =_c \ d \cdot q^ {(k - 1) d}  \log q.
\end {align*}

Next, we compute
\begin{align*}
\log i_{\FF^ k} ^ {\CG} (n) &=
\sum_{\substack {\text {prime powers } q \\ \text {with } |\CG (q) |\leq n}}
\log i_{\FF ^ k} (\CG (q)) \\
& =_c \ d \cdot \sum_{\substack {\text {prime powers } q \\ \text {with }  \frac 1{d\cdot C} q ^ {d} \leq n}}
 q ^ {(k - 1) d} \log q \\
& =_c  \  d \cdot  \left( (d\cdot Cn) ^ {\frac{1}{d}}\right) ^ {(k - 1) d + 1} \text { (by Lemma~\ref {primepowers})}\\
& =_c \  d^ {k +\frac 1d} \cdot n ^ {(k - 1) +\frac{1}{d}},
\end {align*}
This is the explicit upper bound promised.  If now $\CG$ is fixed, then $d$ is a constant, so this bound becomes $ \asymptoticless n ^ {(k - 1) +\frac{1}{d}}.$  

For the lower bound, it suffices to only consider $\CG(p)$ where $p $ is prime.  From Table \ref {groups}, we see that $|\Out(\CG (p) )| $ is bounded for prime $p $.  Therefore, 
\begin {align*}
\log i_{\FF ^ k}(\CG (p)) \asymptoticge \log \left (|\CG (p) | ^ {\frac { |\CG(p) | ^ {k - 1}}{|\Out(\CG (p) )|} }  \right) 
\asymptoticge  d \cdot p^ {(k - 1) d}  \log p.
\end {align*}
The lower bound then proceeds exactly as above, except that the sums are now over primes rather than prime powers.  But as this does not affect the output of Lemma~\ref {primepowers}, we see that $\log i_{\FF^ k} ^ {\CG} (n) \asymptoticge n ^ {(k - 1) +\frac{1}{d}}.$
\end{proof}

We are now ready to prove the main theorem.

\begin {named}{Theorem \ref {maximal}}
Let $\FF ^ k $ be the rank $k $ free group. Then
$$
i_{\FF ^ k}^ {\max\lhd} (n) \asymptotic  e ^ {\left( {n ^ { k -2/3}} \right )}, 
\quad
 i_{\FF ^ k}^ {\max} (n) \asymptotic  i_{\FF ^ k} (n)  \asymptotic e ^ {\left( {n ^ n} \right )}.
$$
\end {named}
\begin {proof}
For  $i_{\FF ^ k}^ {\max} (n)$ and $  i_{\FF ^ k} (n) $, it suffices by Proposition \ref {universalbounds} to prove that $i_{\FF ^ k}^ {\max} (n) \asymptoticge e ^ {\left( {n ^ n} \right )}.$  Since the kernel of any surjection $\FF^k \longrightarrow A (n)$ is the intersection of (maximal) index $n$ subgroups corresponding to the conjugates of $A(n-1) \subset A(n)$, we have by Corollary \ref {fixedgroup} that
$$i_{\FF ^ k}^ {\max} (n) \geq i_{\FF ^ k}(A(n))  \geq (n!/2 ) ^ {\, \epsilon\cdot \frac {(\frac {n!} 2) ^ {k - 1}} 2 } \asymptotic e ^ {n ^ n}.\\ $$

We next show that $i_{\FF ^ k}^ {\max\lhd} (n) \asymptotic  e ^ {\left( {n ^ { k -2/3}} \right )}$.  By Propositions ~\ref{families} and \ref{PSL}, 
$$\log i_{\FF ^ k}^ {\max\lhd} (n) \asymptotic \sum_{\CG} \log i_{\FF ^ k} ^\CG (n) \geq \log i_{\FF ^ k} ^{\PSL_2} (n) \asymptotic n ^ { k -\frac{2}{3}},$$ which gives the lower bound. For the upper bound, first observe that as the five families of types $(1)-(3)$ in Proposition \ref{PSL} have slower intersection growth than $\PSL_2$, removing them from the sum does not change its asymptotics.  Moreover, we only need to sum over type $(4)$ families $\CG$ such that $|\CG(2) | \leq n$, which implies that the dimension $d \leq C\log n$ for some universal $C $.  So by the explicit estimate in the last part of Proposition \ref{PSL}, 
\begin {align*}
\sum_{\CG \text{ with } d\leq C\log n} \log i_{\FF ^ k} ^\CG (n) & \asymptoticless \log i_{\FF ^ k} ^{\PSL_2} (n) \ + \sum_{\substack{\text {Type } (4) \  \CG\neq PSL_2 \\ \text{with } d\leq C\log n}}  C \cdot d^ {k +\frac 1d} \cdot n ^ {(k - 1) +\frac{1}{d}} \\
& \asymptoticless n ^ {k -\frac 23} + \log n \cdot (\log n)^ {k + 1} \cdot n ^ {(k - 1) +\frac 18}\\
&\asymptotic n ^ {k -\frac 23}.
\end {align*}
For the second inequality, a consultation of Table \ref {groups} shows that in the summation we actually have $8 \leq d \leq C \log n$.  Moreover, the number of terms in the sum is at most some constant multiple of $\log n$, which contributes the additional logarithm.
\end{proof}

%(we need to explain this.)
%Something about alternating groups.

\iffalse
Using Equation \ref {families}, we see that
\begin{align*}
\log i_{\FF ^ k} ^ {max\lhd} (n) \ &\asymptotic  \sum_{\substack { \text {single parameter} \\ \text {families } \CG \text { of }\\ \text {finite simple groups}}} i_{F_k} ^ {\CG} (n). \\
&\asymptoticless \small n ^ {k -\frac 23} +  C \log (n) (n\log n) ^ {k -\frac 78}, \left [\substack { \text {\tiny Corollary \ref {estimates} and that }\\ \text {there are only } C\log n \\ \text { non-zero terms above}}\right]\\
&\asymptotic n ^ {k -\frac 23}.
\end {align*}
The reverse inequality follows from the $\PSL_2$ estimate in Corollary \ref {estimates}.
\fi

%Some theorem about the maximal subgroup growth and two paragraph outline of the proof.

\section{Identities in finite simple groups and residual finiteness growth
\label{sec:applications}}

In this section we discuss relations between the intersection growth function, the residual finiteness
growth function and identities in groups.

\begin {defn} An \emph{identity} or \emph{law} on $k$ letters in a group $\Gamma $ is a word
$w(x_1, \ldots , x_k)$ in the free group $\FF^k$ such that $w(g_1, \ldots, g_k) = 1$ for all
$g_1, \ldots, g_k \in \Gamma$.
\end {defn}

Much work has been devoted to study laws in finite groups. 
We only mention a few of them.
Oates and Powell \cite{OP64} and Kov{\'a}cs and Newman \cite{KN66} find
the smallest set of laws which generate every other law in a finite group. Hadad 
\cite{Had11}
finds estimates
on the shortest law in a finite simple group. Laws can also be used to 
characterize classes of groups defined by a set of laws (see the survey by 
Grunewald, Kunyavskii and Plotkin \cite{GKP12}
for the case of solvable groups). A question by Hanna Neumann 
\cite[p. 166]{Neumann}
asks if there is a law which is satisfied in an infinite number of non-isomorphic non-abelian finite simple groups. This question has been answered negatively by Jones \cite{Jon74}.
Our next result can be seen as an answer to the finite
version this question
and follows the same line of arguments used in 
\cite{BM1, KM11} where one finds a law for all finite groups of order at most a given size.
We apply our result on maximal intersection growth of $\FF^k$ to find identities in all
finite simple groups of at most a given size.

\begin{thm}
\label{thm:identities}
For every positive integer $n$, there exists a reduced word $w_n \in \FF^2$ of length
$|w_n|\asymptoticless  n ^ {\frac{4}{3}}$ which is an identity on $2$ letters for all finite
simple groups of size $\le n$.
\end{thm}
\begin{proof}
The proof is essentially an adaptation of the one
of Lemma 5 in \cite{KM11} and we mention it here
for the reader's convenience.

If $B_2(t)$ denotes the size of the ball of radius $t$ within $\FF^2$,
we let $t$ grow until the size of $B_k(t)$ is bigger than $[\FF^2: \Lambda^{max\lhd}_\Gamma(n)]$.
This counting argument shows that one can find
a non-trivial word $w_n \in \Lambda^{max\lhd}_\Gamma(n)$ of length
at most $\log [\FF^2: \Lambda^{max\lhd}_\Gamma(n)] = \log i_{\FF ^ 2}^ {max\lhd} (n) \asymptotic
n ^ {2 -\frac{2}{3}}$.

By construction, the word $w_n$ vanishes for any homomorphism of $\FF^k$ to a finite
simple group $\Gamma$ of size $\le n$ and it is thus an identity on $\Gamma$.
\end{proof}

\begin{rem}
It seems that the result above can be improved following the methods
in \cite{BM1} and \cite{KM11} to build identities as ``long commutators''.
In order to do so, one should use the classification theorem for finite simple groups
to verify that, in any finite simple group of size $\le n$, 
the order of an element is essentially not bigger that $\sqrt[3]{n}$. From this, one should use the methods in Theorem 1.2 in~\cite{BM1} or of
Corollary 11 in \cite{KM11} to show that there exists an identity of length $\asymptotic n$ on $k$
letters which holds for every finite group of size $\le n$. We also refer the reader
to the discussion in Remark 15 in~\cite{KM11}.
\end{rem}

Given a finitely generated, residually finite group $\Gamma=\langle S \rangle$ and $g \in \Gamma$, let
\[
k_\Gamma(g)=\min \{[\Gamma:N]: g \not \in N,  N \unlhd \Gamma\}.
\]
The \emph{residual finiteness growth function} (also called the 
\emph{depth function} in \cite{kharl-myasni-sapir-1}) is the function
\[
F_\Gamma^S(n) = \max_{g \in B^S_{\Gamma}(n)} k_\Gamma(g),
\]
where $B^S_{\Gamma}(n)$ denotes with ball of radius $n$ with respect to the generating set $S$.
The growth rate of $F_\Gamma^S(n)$ is independent of the generating set $S$
(see \cite{Bou}).  The function was first introduced by Bou-Rabee in \cite{Bou}
and subsequently it has been computed in several groups \cite{Bou, Bou2, BM1, BM3, BK11, KM11, kharl-myasni-sapir-1}.

As hinted at in the proof of Theorem \ref{thm:identities}, bounds for $i_{\Gamma}^\lhd(n)$ easily translate to lower bounds for $F^S_{\Gamma}$ using the word growth of $\Gamma$.  Namely,

\begin{observ}
If $\Gamma = \langle S \rangle$ is a finitely generated, residually finite group, $$|B^S_{\Gamma}(k)| > i_\Gamma^\lhd(n) \implies F_\Gamma^S(2k) \geq n.$$
\end{observ}
\begin{proof}
If $|B^S_{\Gamma}(k)| > i_\Gamma^\lhd(n)$ then there is some $w \in B^S_\Gamma(2k)$ that is in every subgroup of $\Gamma $ with index at most $n$.  
\end{proof}

\iffalse%%%%%%%%%%%%%%%%%%%%%%%%%%%%%%%
\begin{rem}
We observe that a precise estimate of the normal intersection growth $i^{\lhd}_{\FF ^ k}(n)$
gives a lower bound on the residual finiteness growth function $F_{\FF ^ k}(n)$. This is, in fact,
what is done in the proof of Lemma 5 in \cite{KM11} and that we adapted in Theorem
\ref{thm:identities} for the case of $i^{\lhd}_{\FF ^ k}(n)$.

In general, the same type of argument yields that, for a finitely generated residually
finite group $\Gamma$, a precise estimate on the word growth function (which measures
the size of balls) and on the normal intersection growth function yield lower bounds on the
residual finiteness growth function $F_{\Gamma}(n)$. More precisely, the following claim is true.
\end{rem}

\begin{claim}
Let $\Gamma$ be a finitely generated residually finite group and let $b_{\Gamma}(n)$
denote the word growth function with respect to a fixed given generating set. Let $f(n)$
denote the smallest number such that $b_{\Gamma}(f(n)-1) < i_{\Gamma}(n)$, but
$b_{\Gamma}(f(n)) \ge i_{\Gamma}(n)$.
Then there exists a word $w_n \in \Lambda_{\Gamma}(n)$ of length at most
$2f(n)$ and so $F_{\Gamma}(|w_n|) > n$.
\end{claim}
\fi%%%%%%%%%%%%%%%%%%%%%%%%%%%%%%%

This, and Theorem \ref{thm:identities} above, lead us to ask the following question:

\begin{quest}
Is it true that $i^{\lhd}_{\FF ^ k}(n) \asymptotic i^{max\lhd}_{\FF ^ k}(n)$? If it is indeed true, then we have the following
two consequences:
\begin{enumerate}
\item The statement of Theorem \ref{thm:identities} can be replaced with one which is true
for all finite groups of order $\le n$ (and not only simple ones).
\item The residual finiteness growth function satisfies $F_{\FF ^ 2}(n) \succeq n^{3/4}$ (which would
improve the current result $F_{\FF^2}(n) \succeq n^{2/3}$ proven in \cite{KM11}).
\end{enumerate}
\end{quest}

\appendix
\section {Number theoretic facts}

The prime number theorem states that $ \pi (n) \congruent n/\log(n),$ where $\pi (n)$ is the number of primes less than or equal to $n$.   We record here some consequences of the prime number theorem used in our asymptotic estimates.

\begin {fact}
If $p_n$ is the $n^{th}$ prime and $q_n$ is the $n^{th}$ prime power, then we have $$p_n \congruent q_n  \congruent n \log n.$$
\end {fact}

\begin {lem}
\label {primepowers}
Suppose $n, l \in \BN $.  Then
$$
\sum_{\substack { \text {prime powers} \\ q\leq n}} q ^ l \log (q) \,  \asymptotic  \, \sum_{\substack { \text {primes } \\ p\leq n}} p ^ l \log (p) \,  \asymptotic  \, n ^ {l + 1}.
$$
\end {lem}

\begin {proof} The first asymptotic equality is just the fact that the $i^{th} $ prime is asymptotic to the $i^{th}$ prime power.  For the asymptotics, first note that
\begin{align*}
\sum_{\substack { \text {primes } \\ p\leq n}} p ^ l \log (p) \asymptotic \sum_{i\leq \frac {n} {\log (n)}} (i\log i)^ l \log (i\log i) 
\asymptotic \sum_{i\leq \frac {n} {\log (n)}} i^l (\log i)^ {l+1}
\end {align*}
So, using integration by parts this is
\begin {align*}
&\asymptotic\left (\frac {n} {\log n}\right) ^ {l + 1}\left (\log\left (\frac n {\log n}\right)\right) ^ {l+1} - \sum_{i\leq \frac {n} {\log (n)}} i^{l+1} \cdot \frac 1i \log(i)^{l}
\end {align*}
However, this latter sum grows slower than $\sum_{i\leq \frac {n} {\log (n)}} i^l (\log i)^ {l+1}$, so
\begin {align*}
\sum_{i\leq \frac {n} {\log (n)}} i^l (\log i)^ {l+1} &\asymptotic \left (\frac n {\log n}\right) ^ {l + 1} \log ^ {l+1} (n) \\
&=n ^ {l + 1}. \qedhere
\end {align*}
\end{proof}

\begin {cor}\label {lcm}
$\lim_{n\to \infty} \frac {\log \lcm( 1, \ldots, n) } n = 1 $, so $\lcm( 1, \ldots, n) \asymptotic e^n$.
\end {cor}

%\end{document}

%\newpage
%\bigskip
%\bigskip
\noindent
Khalid Bou-Rabee \\
Department of Mathematics, University of Michigan \\
2074 East Hall, Ann Arbor, MI 48109-1043, USA \\
{\em E-mail}: \href{mailto:khalidb@umich.edu}{\tt khalidb@umich.edu}\\

\noindent
Ian Biringer \\
Department of Mathematics, Boston College \\
Carney Hall, Chestnut Hill, MA
02467-3806, USA  \\
{\em E-mail}: \href{mailto:ian.biringer@bc.edu}{\tt ian.biringer@bc.edu}\\

\noindent
Martin Kassabov \\
%Department of Mathematics, University of Southampton \\
%University Road, Southampton S017 1BJ, UK \\
%and\\
Department of Mathematics, Cornell University\\
Malott Hall, Ithaca, NY 14850, USA\\
{\em E-mail}: \href{mailto:kassabov@math.cornell.edu}{\tt kassabov@math.cornell.edu}\\

\noindent
Francesco Matucci \\
D\'epartement de Math\'ematiques, Facult\'e des Sciences d'Orsay, \\
Universit\'e Paris-Sud 11, B\^atiment 425, Orsay, France \\
{\em E-mail}: \href{mailto:francesco.matucci@math.u-psud.fr}{\tt francesco.matucci@math.u-psud.fr}
%-----------------------------------------------------------
%-----------------------------------------------------------

\end{document}